\documentclass{article}

\usepackage{a4wide}
\usepackage{cite}
\usepackage[utf8]{inputenc}
\usepackage[T1]{fontenc}

\usepackage[colorlinks=true, pdfstartview=FitV, linkcolor=blue,citecolor=blue, urlcolor=blue]{hyperref}
\usepackage[french,english]{babel}
\usepackage{amsmath}
\usepackage{amscd}
\usepackage{amssymb}
\usepackage{amsrefs}
\usepackage{amsthm}
\usepackage{cases}
\usepackage{framed}
\usepackage{graphicx}
\usepackage{latexsym}
\usepackage{color}
\usepackage{array,multirow,makecell}
\usepackage[capitalise]{cleveref}
\usepackage{comment}
\usepackage{enumitem}

\usepackage{ulem}

\newtheorem{proposition}{Proposition}[section]
\newtheorem{theorem}[proposition]{Theorem}
\newtheorem{corollary}[proposition]{Corollary}
\newtheorem{lemma}[proposition]{Lemma}
\newtheorem*{bootstrap*}{Bootstrap Step}
\theoremstyle{definition}
\newtheorem{definition}[proposition]{Definition}
\newtheorem{remark}[proposition]{Remark}

\numberwithin{equation}{section}

\newcommand{\red}[1]{\textcolor{red}{#1}}

\newcommand\eps{\varepsilon}

\newcommand\de{{\partial}}

\newcommand{\hT}{\mathcal{T}}
\newcommand{\hP}{\mathcal{P}}

\newcommand{\hS}{{\mathcal S}}
\newcommand{\hR}{{\mathcal R}}
\newcommand{\hU}{{\mathcal U}}

\def\R{\mathbb{R}}

\newcommand{\pfp}{\mathfrak{p}^+_f}
\newcommand{\pfm}{\mathfrak{p}^-_f}
\newcommand{\pcp}{\mathfrak{p}^+_c}
\newcommand{\pcm}{\mathfrak{p}^-_c}

\newcommand{\pref}{p^{\text{ref}}}
\newcommand{\uref}{u^{\text{ref}}}

\title{Hard-congestion limit of the p-system in the BV setting}

\author{
Fabio Ancona\footnote{Dipartimento di Matematica "Tullio Levi-Civita", Universit\`a di Padova, Italy, ancona@math.unipd.it}, 
Roberta Bianchini\footnote{Consiglio Nazionale delle Ricerche, Istituto per le Applicazioni del Calcolo, 00185 Rome, Italy, roberta.bianchini@cnr.it},
Charlotte Perrin\footnote{CNRS, Aix Marseille Univ., I2M, Marseille, France,charlotte.perrin@cnrs.fr}
}

\begin{document}

\maketitle
\begin{small}
\begin{abstract}
This note is concerned with the rigorous justification of the so-called \textit{hard congestion limit} from a compressible system with singular pressure towards a mixed compressible-incompressible system modeling partially congested dynamics, for small data in the framework of BV solutions.\\
We present a first convergence result for perturbations of a reference state represented by a single propagating large interface front, while the study of a more general framework where the reference state is constituted by multiple interface fronts is announced in the conclusion and will be the subject of a forthcoming paper.\\
A key element of the proof is the use of a suitably weighted Glimm functional that allows to obtain precise estimates on the BV norm of the front-tracking approximation.

\end{abstract}
\end{small}

\section{Introduction}

In this note, we analyze the following $p-$system expressed in terms of pressure, $p$, and velocity, $u$:
\begin{subnumcases}{\label{eq:psystem-ep}}
\partial_t \hT_\eps (p) - \partial_x u = 0, \\
\partial_t u + \partial_x p = 0,
\end{subnumcases}
where $\tau = \hT_\eps(p)$ is the specific volume, i.e. is the inverse of the density of the fluid, $\varrho = \tau^{-1}$. 
The law $\hT_\eps$ is then defined as the inverse of the pressure law $\hP_\eps=\hP_\eps (\tau)$ that we assume to be a singular function with a vertical asymptot in $\tau=\tau^\star = 1$:
\begin{equation}\label{eq:T-ep}
  \hT_\eps = \hP_\eps^{-1} \quad \text{with, for} ~ \tau > 1, \quad  
  \hP_\eps(\tau) \doteq  \dfrac{\kappa}{\tau^{\gamma_i}} + \dfrac{\eps}{(\tau-1)^{\gamma_c}}, \quad 
  \kappa, \eps > 0,~ \gamma_{i,c} > 1. 
\end{equation}
The functions $\hP_\eps$ and $\hT_\eps$ are plotted in Figure~\ref{fig:pressure}.
This system is a reformulation in Lagrangian coordinates of the one-dimensionnal compressible Euler equations:
\[
\begin{cases}
\partial_t \varrho + \partial_y (\varrho u) = 0, \\
\partial_t(\varrho u) + \partial_y(\varrho u^2) + \partial_y \tilde{\mathcal{P}}_\eps(\varrho) = 0.
\end{cases}
\]

\begin{figure}
	\centering {\includegraphics[scale=0.45]{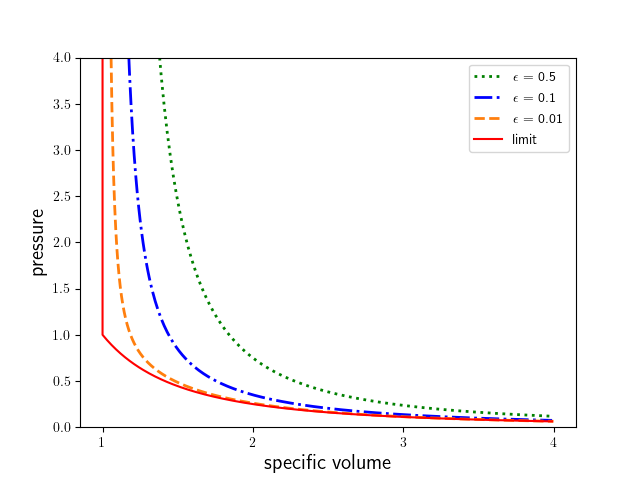}
 \includegraphics[scale=0.45]{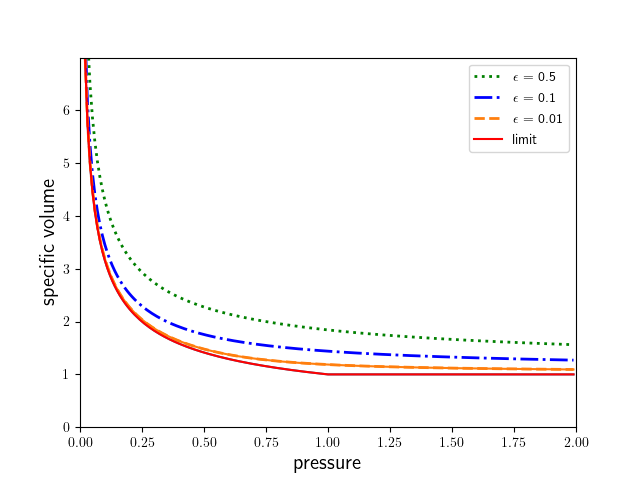}}
	\caption{Behavior of the pressure $\hP_\eps(\tau)$ (on the left) and its inverse $\hT_\eps$ (on the right) as $\eps \to 0$.}
	\label{fig:pressure}
\end{figure}

We are interested in the rigorous justification of the limit as $\eps \to 0$ towards the target system
\begin{subnumcases}{\label{eq:psystem-lim}}
\partial_t \hT(p) - \partial_x u = 0, \\
\partial_t u + \partial_x p = 0,
\end{subnumcases}
where 
\begin{equation}\label{eq:T-lim}
\hT(p) = \begin{cases}
        \big(\kappa p^{-1}\big)^{1/\gamma_i} =: \hT^{i}(p) \quad & \text{if} \quad  p < \kappa, \\
        1 \quad & \text{if} \quad p \geq \kappa.
\end{cases}
\end{equation}
The domain $\{\tau =1\}$ is said to be \textit{congested}.
The systems~\eqref{eq:psystem-ep} and~\eqref{eq:psystem-lim} are indeed used to model congestion, or saturation, effects. 
At the level of the approximate system~\eqref{eq:psystem-ep}-\eqref{eq:T-ep}, the singular pressure law $\hP_\eps$ (or equivalently $\tilde\hP_\eps(\varrho)\doteq \hP_\eps(\varrho^{-1})$) models some repulsive forces preventing the specific volume $\tau$ to take values below the threshold $\tau^\star = 1$, like standard penalty methods. 
This type of model is called in the literature ({\it c.f.}~\cite{maury2011}) \textit{soft congestion model}.
At the level of the limit system~\eqref{eq:psystem-lim}, the specific volume law $\hT$ saturates, i.e. attains its minimal value, at the value $\tau^\star$. 
This type of model is called in the literature \textit{hard congestion model}.

These types of models (or their reformulations) can be used in numerous contexts.
\begin{itemize}
    \item \textit{Modeling of mixtures}: The Eulerian version of~\eqref{eq:psystem-lim} has been originally derived in~\cite{bouchut2000}.
    It is seen as an asymptotic model of biphasic liquid-gas equations when the ratio between the reference gas and liquid densities tends to $0$.\\
    At the approximate level, $\eps > 0$, singular pressure laws such as $\hP_\eps$ are sometimes called \textit{hard sphere potentials} and used for granular mixtures (see for instance~\cite{carnahan1969}).

    \item \textit{Modeling of crowds and vehicular traffic}: In this context, it is natural to consider singular pressure laws to model short-range repulsive social interactions, see for instance~\cite{degond2011},~\cite{maury2011} or~\cite{bresch2017}. \\
    In the context of vehicular traffic, the famous Aw-Rascle-Zhang (ARZ) model can also include such singular potential in the offset function as proposed in Berthelin et al~\cite{BDDR}.
    The hard congestion version of the ARZ model has been justified through a particle approximation (so called Follow-The-Leader model) by Berthelin and Goatin in~\cite{berthelin2017FTL}. 
    The limit $\eps \to 0$ between the soft congestion ARZ and the hard congestion ARZ systems has been studied numerically in Berthelin et al.~\cite{berthelin2017}.  

    \item \textit{Modeling of wave-structure interactions and partially free surface flows}: Several systems have been recently derived and analyzed to model flows in closed pipes in the one hand (see~\cite{bourdarias2012} and~\cite{godlewski2018}), and floating objects on the other hand (see for instance~\cite{godlewski2020}, \cite{godlewski2018}).
    There is indeed a structural analogy between compressible Euler equations (and so~\eqref{eq:psystem-lim}) and the Shallow Water system by identifying $\varrho$ with $h$, the height of the flow. 
    In the above studies, the height of the flow is constrained by a roof of a channel or a floating structure. 
    The isentropic component $\kappa \varrho^{\gamma_i}$ appearing in~\eqref{eq:T-ep} can be then understood as the hydrostatic pressure in the Shallow Water equations.
\end{itemize}
From the mathematical standpoint, the singular limit $\eps \to 0$ from soft-congestion systems towards hard congestion systems has been previously studied in various frameworks, in particular: 
a strong $\mathcal{C}^1$ (local-in-time) setting in~\cite{bianchini2021}, and a weak setting (namely global-in-time finite energy weak solutions) when additional viscosity is taken into account, see for instance~\cite{perrin2015}, or at the level of the Riemann problems in~\cite{degond2011} (Appendix A). 
Weak solutions to the Eulerian version of~\eqref{eq:psystem-lim}, when $\kappa=0$, have been constructed by other means: through a discrete (sticky blocks) approximation in~\cite{berthelin2002}, via an convex optimization point of view in~\cite{perrin2018}. 
The interested reader is referred to the survey paper~\cite{perrin2018-survey} for more references on the subject.

\medskip
In this note, we present a convergence result for (global-in-time) solutions with BV regularity. 
This framework is, to some extent, natural in view of the hyperbolic nature of system~\eqref{eq:psystem-ep} and the general theory developed by
Di Perna, Risebro, Bressan and collaborators (see~\cite{bressan2000} and references therein) around the construction of such solutions via \textit{Wave Front Tracking} algorithms for hyperbolic systems of conservation laws.

It is natural in the BV setting to look for solutions lying in a neighborhood of a reference partially congested solution.
Hence a given reference solution $(\pref,\uref)$ of~\eqref{eq:psystem-lim} is approximated at $\eps > 0$ by  a solution $(\pref_\eps,\uref_\eps)$ to~\eqref{eq:psystem-ep}. 
Our goal is to construct BV solutions $(p_\eps,u_\eps)$ in the neighborhood of $(\pref_\eps,\uref_\eps)$ via the wave front tracking method and, for $\eps \to 0$, to extract a subsequence $(p_\eps,u_\eps)_\eps$ converging weakly towards a weak entropy (partially congested) solution $(p,u)$ of the limit system~\eqref{eq:psystem-lim}.
Moreover, we aim at characterizing the solution in the congested domain and the dynamics of the interface between the free domain and the congested one.

\medskip
Our result is the first global-in-time convergence result of the "hard congestion" type in an inviscid setting.
To some extent, the singular limit $\eps \to 0$ shares similarities with the famous low Mach number limit characterizing the transition between a compressible regime and an incompressible one (see for instance~\cite{alazard2008}).
Indeed system~\eqref{eq:psystem-ep} corresponds to a fully compressible system, while system~\eqref{eq:psystem-lim} is a mixed compressible/incompressible system since incompressibility condition, $\partial_x u =0$, holds in the congested domain where $\tau =1$.
Wave Front Tracking methods have been previously used for studying the low Mach number limit.
In particular, in the works of Colombo, Guerra and Schleper~\cite{colombo2016ARMA, colombo2016JHDE}, the low Mach number limit in a biphasic system with separated  gas (compressible) and liquid (nearly incompressible) phases, is analyzed.
Our study relies on a similar formulation in terms of pressure and velocity (instead specific volume/velocity for the classical p-system), and faces the same type of difficulties related to the unbounded speed of propagation of waves in the incompressible (congested) phase. 
However it strongly differs from~\cite{colombo2016ARMA, colombo2016JHDE} on the treatment of the interface between the compressible/free phase and the incompressible/congested phase.
In~\cite{colombo2016ARMA, colombo2016JHDE}, the gas and liquid phases are supposed to be immiscible, which means that there is no mass exchange between the two regions and the volume of the liquid, incompressible, phase remains constant.
From the mathematical point of view, it means that the interface is stationary in the Lagrangian mass coordinates.
This property is not satisfied by the free-congested system~\eqref{eq:psystem-lim}. 
There are obviously mass exchanges between the free region and the congested one, and the interface is a free boundary that is related to the trace of the solution from both sides of the interface (see~\eqref{eq:interface} below, and the study~\cite{iguchi2019} on the notion of fully nonlinear boundary condition). 
The analysis of the evolution and interactions of the free/congested interfaces in the approximate system and in the limit hard congestion model is one of the main novelties of our approach, which relies on sharp BV estimates
and on appropriate rescaling of the singular pressure law and of the corresponding specific volume function in the congested region.

\medskip
Let us now describe more in details our framework and the main results. We will be tracking the evolution of \textit{small BV perturbations} of a reference state constituted by a single propagating \textit{large interface} front.

\paragraph{Reference solution formed by a single discontinuity interface and functional setting.}
We choose $p_{0,1}, p_{0,2}$, $u_{0,1}, u_{0,2}$ such that
\begin{equation}\label{df:U0ref}
p_{0,2} < \kappa < p_{0,1}, \qquad
u_{0,2} < u_{0,1},
\end{equation}
and
\begin{equation}\label{eq:RH-lim-ref}
p_{0,2} - p_{0,1} = - \dfrac{(u_{0,2} - u_{0,1})^2}{1- \hT^i(p_{0,2})}.
\end{equation}
Such conditions ensure that $(p_{0,2},u_{0,2})$ is a free state that is the right state of a discontinuity interface of the second family with left state $(p_{0,1},u_{0,1})$:
\begin{equation}
\label{eq:ref-sol-1}
(\pref,\uref)(t,x) \doteq 
\begin{cases}
(p_{0,1},u_{0,1}) & \quad \text{if} \quad x < \bar \lambda_2 t, \\
(p_{0,2},u_{0,2}) & \quad \text{if} \quad x > \bar \lambda_2 t,
\end{cases}
\end{equation}
The reference solution \eqref{eq:ref-sol-1} is a discontinuity interface, which is partially congested on the left and travels to the right with speed
\begin{equation}\label{RH-lim}
\bar \lambda_2 \doteq  - \dfrac{u_{0,2} - u_{0,1}}{\hT(p_{0,2}) - 1} > 0.
\end{equation}
{\it Approximated reference solution.}
The first step is to construct the initial data of the approximate model \eqref{eq:psystem-ep} as a suitable approximation (in $\eps$) of the original reference solution \eqref{eq:ref-sol-1}.
For fixed $\eps > 0$, we choose $(p_{0,1}^\eps, u_{0,1}^\eps)$ and $(p_{0,2}^\eps, u_{0,2}^\eps)$ as follows.
\begin{align*}
  p_{0,1}^\eps = p_{0,1}> \kappa, \\
  p_{0,2}^\eps < \kappa, \quad \text{with} \quad p_{0,2}^\eps \to p_{0,2} \text{ as } \eps \to 0,
\end{align*}
and
\[
u_{0,1}^\eps = u_{0,1} > u_{0,2}^\eps,
\]
such that
\begin{equation}\label{eq:RH-eps-ref}
p_{0,2}^\eps - p_{0,1}^\eps = - \dfrac{(u_{0,2}^\eps - u_{0,1}^\eps)^2}{\hT_\eps(p_{0,1}^\eps)- \hT_\eps(p_{0,2}^\eps)}.
\end{equation}
We observe that the above conditions, together with~\eqref{eq:T-ep}, imply that $u_{0,2}^\eps \to u_{0,2}$ as $\eps \to 0$.

\noindent The initial data are represented in Figure~\ref{fig:sol-ref}.
\begin{figure}
    \centering
    \includegraphics[scale=0.7]{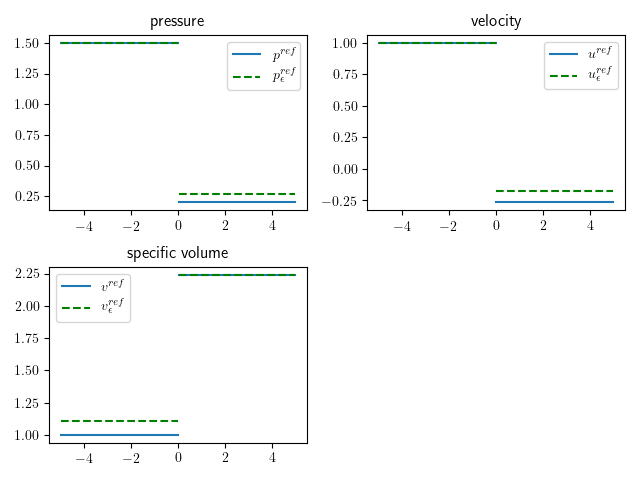}
    \caption{The reference solution is a 2-shock interface. Plain line: limit reference solution; dotted line: $\eps$-reference solution.}
    \label{fig:sol-ref}
\end{figure}

\bigskip
{\it Notation and convention} Here, and throughout the paper, we use the following notation.
\begin{itemize}
\item For a function $g\in {\bf BV}(\R)$, and for any $-\infty\leq a<b\leq +\infty$, introducing the notation
\begin{equation*}
    \widetilde g(x):=
    \begin{cases}
    \displaystyle{\lim_{x\to a+}}g(x)
    \quad &\text{if}\qquad x\leq a,\\
    \noalign{\smallskip}
    g(x)\quad &\text{if}\qquad a<x<b,\\
     \noalign{\smallskip}
        \displaystyle{\lim_{x\to b-}}g(x)
\quad &\text{if}\qquad x\geq b,
    \end{cases}
\end{equation*}
we denote
\begin{equation}
    \text{TV}(g, \, [a,b])\doteq \text{TV}(\,\widetilde g\,),
\end{equation}
 the total variation of the restriction of the function $g$ to the interval $[a,b]$.
Moreover, we will use the notation
$g(a\pm)\doteq \lim_{x\to a\pm}g(x)$
for the one-side limits of $g$ in $a$.
For any set $A$ we let $\overset{\ \circ}{A}$ denote the interior of $A$.
We will also use the notation
$\R^+\doteq \,]0, +\infty[$.
\item The notation $\sigma_\alpha:=p_r-p_\ell$ denotes a wave located at $x_\alpha (t)$ and connecting the left and right states $\hU^\ell=(p_\ell, u_\ell), \, \hU^r=(p_r, u_r)$.
\end{itemize}

\bigskip
{\it Functional setting.}
Given $\delta, \eps > 0$ and a curve $(\bar x^\eps(t))_{t \geq 0}$, with $\bar x^\eps(0) = 0$, we set
\begin{align}
    \Omega^{\delta,\eps}_i &\doteq \ ]p_{0,i}^\eps -\delta, p_{0,i}^\eps + \delta[ \ \times \ ]u_{0,i}^\eps -\delta, u_{0,i}^\eps + \delta[ \ , \quad i=1,2, \notag\\
    I_{1,t}^\eps &\doteq ]-\infty, \bar x^\eps(t)], \quad I_{2,t}^\eps \doteq [\bar x^\eps(t), +\infty[ \ .
\end{align}
At $\eps$ fixed, we will look for solutions $(p^\eps(t, \cdot), u^\eps(t,\cdot))$ of~\eqref{eq:psystem-ep} belonging for a time $t\geq 0$ to the domain
\begin{equation}
\label{eq:set-solns-eps}
\begin{aligned}
    \mathcal D^{\delta,\eps}_t&\doteq 
    \Big\{
    (p, u)
    \in {\bf BV}(\R; \R^+\times\R)\, :\
    (p, u) (x)\in\ \Omega_i^{\delta,\eps}\quad
    \forall~x\in \overset{\circ}{I^\eps_{i,t}},\\
    &\hspace{1.8in}
    \sum_{i=1}^2\text{TV}\big(p,\, I^\eps_{i,t}\big)
    + \frac{1}{\eps^{\frac{1}{2\gamma}}}\,\text{TV}\big(u,\, I^\eps_{1,t}\big)
    + \text{TV}\big(u,\, I^\eps_{2,t}  \big)
    <\delta
    \Big\}.
\end{aligned}
\end{equation}
As $\eps \to 0$, the limit $(p(t,\cdot), u(t,\cdot))$ is expected to belong to the domain 
\begin{equation}
\label{eq:set-solns-lim}
\begin{aligned}
    \mathcal D^{\delta}_t&\doteq 
    \Big\{
    (p, u)
    \in {\bf BV}(\R; \R^+\times\R):\
    (p, u) (x)\in\ \Omega_i^{\delta} ~~
    \forall~x\in \overset{\circ}{I_{i,t}}, 
    ~u(x) \equiv u^c =\text{const.}\; \forall x \in \overset{\circ}{I_{1,t}}\\
    &\hspace{1.8in}
    \sum_{i=1}^2\text{TV}\big(p,\, I_{i,t}\big)
    + \text{TV}\big(u,\, I_{2,t}  \big)
    <\delta
    \Big\},
\end{aligned}
\end{equation}
where 
\begin{align}
    \Omega^{\delta}_i \doteq \ ]p_{0,i} -\delta, p_{0,i} + \delta[ \ \times \ ]u_{0,i} -\delta, u_{0,i} + \delta[ \ , \quad i=1,2, \\
    I_{1,t} \doteq ]-\infty, \bar x(t)], \quad I_{2,t} \doteq [\bar x(t), +\infty[ \ .
\end{align}

\medskip
\paragraph{Definition of entropy weak solutions and main results.}
As usual in the context of weak distributional solutions, we will be asking a stronger characterization in terms of suitable \textit{admissibility conditions} to be satisfied by the weak distributional solutions of~\eqref{eq:psystem-ep} and~\eqref{eq:psystem-lim}.
First, we introduce the notion of pair of entropy/entropy flux.
\begin{definition}
A continuously differentiable function
$\eta=\eta(\tau,u)$ is called:
\begin{itemize}
    \item[-] an entropy for~\eqref{eq:psystem-ep}, with entropy flux $q=q(p,u)$ being a continuously differentiable function, if there holds
\begin{equation}
\partial_\tau\eta(\hT_{\eps}(p), u)= - \partial_u q(p,u),\qquad\quad
\partial_u\eta(\hT_{\eps}(p), u)= \partial_p q(p,u)
\qquad\forall~p\in\R^+, u\in \R\,;
\end{equation}
 \item[-] an entropy for~\eqref{eq:psystem-lim}, with entropy flux $q=q(p,u)$ being
a continuously differentiable function, if there holds
\begin{equation}
\partial_\tau \eta(\hT(p), u)= - \partial_u q(p,u),\qquad\quad
\partial_u\eta(\hT(p), u)= \partial_p q(p,u)\qquad\forall~p\in\R^+, u\in \R\,.
\end{equation}
\end{itemize}
A pair $(\eta(v,u), q(p,u))$ is called a convex entropy/entropy flux pair for ~\eqref{eq:psystem-ep} (resp. for for~\eqref{eq:psystem-lim})
if $\eta$ is a convex map and it is an entropy for~\eqref{eq:psystem-ep} (resp. for~\eqref{eq:psystem-lim}), with associated entropy flux~$q$.
\end{definition}
\begin{definition}[Entropy weak solution of soft-congestion p-system]
\label{entr-sol-def-1}
A function 
\[
\hU\doteq (p, u):[0,+\infty[\,\times\,\R\to
\R^+\times \R,
\]
is said an entropy weak solution of 
the Cauchy problem for~\eqref{eq:psystem-ep}
with initial datum~$(p_\text{in}, u_\text{in})
    \in {\bf BV}(\R; \R^+\times~\R)$
if the following holds:
\begin{enumerate}[label=(\roman*)]
    \item the map $t\mapsto \hU(t)$, $t\geq 0$, is continuous 
    as a function with values in ${\bf L^1_{loc}}(\R; \R^+\times~\R)$ and the initial condition
    \begin{equation}\label{eq:initial-datum}
    (p,u)(0, \cdot)=(p_\text{in}, u_\text{in}),
    \end{equation}
    is satisfied.
    \item $\hU$ is a weak distributional solution of~\eqref{eq:psystem-ep} on $\R^+\times\R$, that is, for any test function $\varphi\in\mathcal{C}^1_c$ with compact support contained in $\R^+\times\R$, there holds
    \begin{equation*}
        \int_{\R^+}\int_{\R}
        \big\{\hT_{\eps}(p)\,\partial_t\varphi-u\,\partial_x \varphi\big\}dx dt =0\,,
        \qquad\quad
        \int_{\R^+}\int_{\R}
        \big\{u\,\partial_t \varphi +p\,\partial_x \varphi\big\}dx dt =0\,.
         \end{equation*}
         \item for every pair $(\eta, q)$ of convex entropy/entropy flux for ~\eqref{eq:psystem-ep}, and          for any non-negative test function $\varphi\in\mathcal{C}^1_c$ with compact support contained in $\R^+\times\R$, there holds
        \begin{equation*}
        \int_{\R^+}\int_{\R}
        \big\{\eta\big(\hT_{\eps}(p), u\big)\,\partial_t \varphi
        + q(p,u)\,\partial_x \varphi\big\}dx dt \geq 0\,.
         \end{equation*}
\end{enumerate}
\end{definition}
\begin{definition}[Entropy weak solution of the hard-congestion p-system]
\label{entr-sol-def-2}
A function 
\[
\hU\doteq (p, u):[0,+\infty[\,\times\,\R\to\R^+\times \R,
\]
is said an entropy weak solution of 
the Cauchy problem for~\eqref{eq:psystem-lim} with initial datum~$(p_\text{in}, u_\text{in})\in \mathcal D^{\delta}_0$
if there holds:
\begin{enumerate}[label=(\roman*)]

         \item{\label{item:ID}} the map $t\mapsto \hU(t)$, $t\geq 0$, is continuous as a function with values in ${\bf L^1_{loc}}(\R; \R^+\times~\R)$ and it satisfies the initial condition
 
         \begin{equation}
         \label{eq:initial-datum-3}
        (p,u)(0,\cdot)=\text{In}(p_\text{in}, u_\text{in}), 
        \end{equation}
        where 
        \begin{equation}\label{eq:in-function}
        \text{In}(p_\text{in}, u_\text{in})(x)
        \doteq
        \begin{cases}
        (p_\text{in}, u_\text{in})(x)\qquad &\text{if}\quad x >\overline x(0), \\
        \noalign{\smallskip}
        \big(p_\text{in}^c,u_\text{in}^c\big) \qquad &\text{if}\quad x <\overline x(0),
        \end{cases}
        \end{equation}
with 
        \begin{equation}
         p_\text{in}^c 
         \doteq
         p_\text{in}(\overline{x}(0)+)
         + \dfrac{\big(u_\text{in}^c - u_\text{in}(\overline{x}(0)+)\big)^2}
         {\hT(p_\text{in}(\overline{x}(0)+)) - 1},
        \end{equation}
        and where $u_\text{in}^c$ is the constant value taken by the initial velocity $u_\text{in}$ in the congested domain as in \eqref{eq:set-solns-eps}.
    \item
    $\hU$ is a weak distributional solution of~\eqref{eq:psystem-lim} on $\R^+\times\R$, that is, for any test function $\varphi\in\mathcal{C}^1_c$ with compact support contained in $\R^+\times\R$, there holds
    \begin{equation*}
        \int_{\R^+}\int_{\R}
        \big\{\hT(p)\,\partial_t \varphi-u\,\partial_x \varphi\big\}dx dt =0\,,
        \qquad\quad
        \int_{\R^+}\int_{\R}
        \big\{u\,\partial_t \varphi+p\,\partial_x \varphi\big\}dx dt =0\,.
         \end{equation*}
         \item For every pair $(\eta, q)$ of convex entropy/entropy flux for~\eqref{eq:psystem-lim}, and for any non-negative test function $\varphi\in\mathcal{C}^1_c$ with compact support contained in $\R^+\times\R$, there holds
        \begin{equation*}
        \int_{\R^+}\int_{\R}
        \big\{\eta\big(\hT(p), u\big)\,\partial_t \varphi
        + q(p,u)\,\partial_x \varphi\big\}dx dt \geq 0\,.
         \end{equation*}
\end{enumerate}
\end{definition}

\medskip

\begin{remark}\label{rmk:ID-hard}
Notice from~\ref{item:ID} in Definition \ref{entr-sol-def-2} that the initial pressure $p_\text{in}$ is in general not attained by the solution in the congested region, because of the infinite propagation speed of waves of the congested domain, see Remark~\ref{rmk:inf-wave} below. This is the reason why the initial datum needs to be redefined by means of the function $\text{In}(p_\text{in}, u_\text{in})$ in \eqref{eq:in-function}.
The new congested initial datum $(p^c_\text{in}, u^c_\text{in})$ provided by such function $\text{In}(p_\text{in}, u_\text{in})$
is precisely the left 
state of a 2-discontinuity interface with right 
state $(p^r,u^r) \doteq 
\big(p_\text{in}(\overline x(0)+), 
u_\text{in}(\overline x(0)+) \big)$.
The intermediate state $(p^m,u^m) \doteq \big(p_\text{in}^c, u_\text{in}(\overline{x}(0)-)\big)$ can be seen as the unique state such that:
\begin{itemize}
     \item[-] $(p^m,u^m)$ is the right state of a 1-wave with infinite speed with left state \\ $(p^l,u^l)\doteq \big(p_\text{in}(\overline x(0)-),     u_\text{in}(\overline x(0)-)\big)$;
     \item[-]
     $(p^m,u^m)$ is the left state of a 2-discontinuity interface with right state\\ $(p^r,u^r)\doteq \big(p_\text{in}(\overline x(0)+), u_\text{in}(\overline x(0)+)\big)$. 
\end{itemize}
It is also important to note that if two states  are connected by a wave of infinite speed, then they have the same velocity (the velocity is constant in the congested domain).
This is why we only need to define the congested pressure $p_\text{in}^c$. Indeed for the velocity we simply have $u^m = u^l = u_\text{in}(\overline{x}(0)-)=u^c_\text{in}$.
\end{remark}

\medskip

We state now our main result.
\begin{theorem}\label{thm:main}
Let $(p^\text{ref},u^\text{ref})$ satisfy~\eqref{df:U0ref}-\eqref{RH-lim}, and $\delta_0, \overline\eps_0>0$.
\begin{itemize}
    \item There exist constants $0< \delta_1<\delta_0$, $0<\overline\eps_1<\overline\eps_0$, so that, for any  $0<\eps<\overline\eps_1$, and for every initial datum $(p_\text{in}, u_\text{in})$ in the domain  $\mathcal{D}^{\delta_1,\eps}_0$ defined as in~\eqref{eq:set-solns-eps}, the Cauchy problem~\eqref{eq:psystem-ep}-\eqref{eq:initial-datum} admits an entropy weak solution $\hU^\eps = (p^\eps, u^\eps)$ in the sense of Definition~\ref{entr-sol-def-1} which satisfies $\hU^\eps(t) \in \mathcal{D}^{\delta_0,\eps}_t$.
    Moreover, there exists a Lipschitz continuous curve $\overline x^\eps : [0,+\infty)\to \R$, with $\overline x^\eps(0)= 0$, representing the interface between the congested domain $I^\eps_{1,t}$ and the free domain $I^\eps_{2,t}$, such that:
    \begin{equation}
    \label{eq:interf-speed-bound}
        \bar \lambda_2 - \delta_0 \leq \dot{\overline x}^\eps(t) \leq \bar \lambda_2 + \delta_0, \qquad\text{for \ a.e.}\ t>0,\quad \forall~\eps>0\,
    \end{equation}
    where $\bar \lambda_2$ is the speed of the reference propagating front defined in~\eqref{RH-lim}.
    \item  Consider an initial datum $\hU_\text{in}\doteq (p_\text{in}, u_\text{in})$ in the domain $\mathcal{D}^{\delta_1}_0$ defined as in~\eqref{eq:set-solns-lim}, and let $\hU^\eps_\text{in}\doteq (p_\text{in}^\eps, u_\text{in}^\eps)\in \mathcal{D}^{\delta_1,\eps}_0$, $\eps > 0$, be initial data such that
    \begin{equation}\label{eq:data-eps-approx-limdatum}
    \hU^\eps_\text{in}\to \hU_\text{in} \qquad\text{in}\qquad {\bf L^1_{loc}}(\R;\,\R^+\times\R) \qquad\text{as}\quad\eps\to 0\,.
    \end{equation}
    Let $\hU^\eps = (p^\eps, u^\eps)$ denote the entropy weak solutions of the Cauchy problem~\eqref{eq:psystem-ep} with initial datum~$\hU^\eps_\text{in}$, and $\overline x^\eps: [0,+\infty)\to \R$, with $\overline x^\eps(0)=0$, denote the corresponding discontinuity interface.
    Then there exist functions
    \begin{equation}
    \hU^*\doteq (p^*, u^*):[0,+\infty[\,\times\,\R\to \R^+\times \R\,,\qquad
    \overline x:[0,+\infty[\,\to\R\,,
    \end{equation}
    so that
    \begin{itemize}
    \item[(i)] 
    up to a subsequence, as $\eps\to 0$, one has:
    \begin{equation}
    \begin{aligned}
    \label{eq:conv-gamma-eps-1}
        \qquad
        \overline x^\eps \to \overline x \qquad &\text{uniformly on every}\ \  [0, T], \, T>0,\\
        \dot{\overline x}^\eps(t) \to \dot{\overline x}(t) \qquad &\text{for a.e.}\quad t\in\R^+.
    \end{aligned}
    \end{equation}
    \smallskip
    \begin{equation}\label{eq:conv-u-eps}
        u^\eps \to u^* \qquad \text{strongly in}\qquad {\bf L^1_{loc}}(\R_+ \times \R; \R),
    \end{equation}
  \begin{equation}
    \label{eq:conv-p-}
     p^\eps(\cdot,\cdot)
          \rightharpoonup p^*(\cdot, \cdot)
          \qquad \text{weakly-* in}\qquad
          {\bf L^\infty}(\R_+ \times \R; \R_+),
    \end{equation}
    and $\dot{\bar x}$ satisfies inequality~\eqref{eq:interf-speed-bound}.

    \item[(ii)] 
    $\hU^*$ is an entropy weak solution of the Cauchy problem~\eqref{eq:psystem-lim}-\eqref{eq:initial-datum}
    in the sense of Definition~\ref{entr-sol-def-2} and
      \begin{equation}
    \label{eq:tv-u-p-v-4}
        \hU^*(t)\in 
    \mathcal{D}^{\delta_2}_t\qquad\forall~ t>0\,.
    \end{equation}
    
      \item[(iii)] In the congested domain, i.e. for a.e. $(t,x) \in \underset{t \geq 0}{\bigcup} \{t\} \times I_{1,t}$, the dynamics is given by:
      \begin{align}
      u^*(t,x)&=  u_c(0), \\
      \dot{\bar x}(t) & = \dfrac{u_c(0) - u(t, \bar x(t)+)}{\tau(t,\bar x(t) +) - 1}, \label{eq:interface}\\
      p^*(t,x)&= p_c(t) = p(t,\bar x(t)+) +\dfrac{\Big( u_c(0) - u(t, \bar x(t)+) \Big)^2}{\tau(t,\bar x(t) +) - 1}.
      \end{align}
\end{itemize}
\end{itemize}
\end{theorem}
\bigskip
\begin{remark} 
\begin{itemize}
    \item As said in the introduction, we rely on a reformulation of the p-system in $(p,u)$, similarly to Colombo and Guerra in~\cite{colombo2016JHDE}.
    In that latter, the use of the pressure variable (instead of the specific volume) has the clear advantage that the pressure and the velocity are  continuous across the (fixed, stationary) interface. 
    This is not the case in the present study, all the quantities are indeed discontinuous across the moving interface $\bar x(t)$, which can be understood as a limit 2-shock wave. 
    This makes the identification of the limit traces across the interface very different from~\cite{colombo2016JHDE}.
    \item The use of the pressure variable, instead of the specific volume, allows us to characterize (see Definition~\ref{df:nature}) and describe in a quantitative manner (with respect to $\eps$) the dynamics of the congested zone, but more importantly, it allows a simple characterization of the limit solutions captured by limit $\eps \to 0$.
    Previously, \textit{hard congestion systems} have been set in terms of specific volume and velocity (see the models introduced in~\cite{perrin2018-survey}). 
    The congestion constraint was expressed through the following unilateral constraint (macroscopic version of the well-known \textit{Signorini's conditions} for contact problems):
    \begin{equation}
    \tau \geq 1, \quad (\tau - 1)\pi = 0, \quad \pi \geq 0,
    \end{equation}
    where the limit congestion pressure $\pi = \lim_\eps \big[\eps (\tau_\eps -1)^{-\gamma_c}\big]$ is seen as the Lagrange multiplier associated to constraint on the velocity field in the congested domain.
    In the formulation~\eqref{eq:psystem-lim}, the congestion constraint on the specific volume is explicitly included in the (inverse) equation-of-state $\hT$~\eqref{eq:T-lim}, and we can formulate a natural entropy criteria to select weak solutions (see Definition~\ref{entr-sol-def-2}).
    Note that the presence of an isentropic component in the pressure (i.e. $\kappa > 0$) is crucial to express the system in $(p,u)$, since it allows to express the specific volume as a function of the pressure in the free domain.
    \item Let us finally observe that the present setting and the use of the Wave Front Tracking algorithm allows us to "track" the dynamics of the interface for all times $t\geq 0$. 
    This is a strong improvement compared to the previous work of Bianchini \& Perrin~\cite{bianchini2021} where the limit $\eps \to 0$ is tackled by means of compactness methods and a weak ($L^1$) control of the pressure on a limited time interval (independent of $\eps$).
    The study of global-in-time solutions is also a important difference with the study of Iguchi and Lannes~\cite{iguchi2019} on the floating body problem. 
\end{itemize}
\end{remark}
\bigskip
In the rest of the notes, we give the key elements of the proof of Theorem~\ref{thm:main}.
In Section~\ref{sec:WFT}, we explain how to construct approximate wave front tracking solutions, $\hU^{\eps,\rho}$, that are piecewise constant for fixed $\rho > 0$.
In particular, we introduce the key weighted Glimm functional which yields important Lipschitz properties (Section~\ref{sec:propWFT}) to pass to the limit: first the WFT limit, $\rho \to 0$, in Section~\ref{sec:lim-rho} leading to the existence of entropy weak solutions $\hU^\eps$ to system~\eqref{eq:psystem-ep}, then the hard congestion limit, $\eps \to 0$, in Section~\ref{sec:lim-ep} leading to the existence of entropy weak solutions $\hU$ to system~\eqref{eq:psystem-lim}.
We finally present in Section~\ref{sec:concl} some extensions of Theorem~\ref{thm:main} to more complicated reference solutions $\hU^{ref}$ that are constituted by more than one interface.

\bigskip
For sake of brevity and clarity, we have made the choice to present only the key elements of the proof of Theorem~\ref{thm:main}, elements that are the most representative of the features of system~\eqref{eq:psystem-ep} and the singularities appearing in the hard congestion limit as $\eps \to 0$.
The interested reader will find in the forthcoming paper~\cite{ancona2022} a more detailed description of the proofs.

\section{Approximate solution constructed by a Wave Front Tracking algorithm}\label{sec:WFT}
A key advantage of working in pressure-velocity variables $(p,u)$ is an easy characterization of the nature of the states, which is provided below.
\begin{definition}[Nature of the states]\label{df:nature}
Let $\pfm,\pfp,\pcm,\pcp > 0$ be independent of $\eps$, and
\begin{equation*}
    \pfm < \pfp < \kappa < \pcm < \pcp.
\end{equation*}
Take $\delta_0, \bar\eps_0$ in Theorem~\ref{thm:main} such that
\begin{equation}
    \begin{cases}
    \pfm < p_{0,2}^\eps - \delta_0 < p_{0,2}^\eps + \delta_0 < \pfp ,\\
    \pcm < p_{0,1}^\eps - \delta_0 < p_{0,1}^\eps + \delta_0 < \pcp,
    \end{cases}
    \quad \text{for all} ~ \eps < \bar\eps_0 .
\end{equation}
A state $\hU =(p,u)$ is called
\begin{itemize}
    \item a \textit{free state}, denoted (F), if $p \in [\pfm,\pfp]$;
    \item a \textit{congested state}, denoted (C), if $p \in [\pcm,\pcp]$.
\end{itemize}
\end{definition}
\bigskip

\paragraph{Hyperbolicity of the soft congestion system.}
Here we collect some properties of the soft congested system \eqref{eq:psystem-ep}.
\begin{proposition}
System~\eqref{eq:psystem-ep} is strictly hyperbolic on the domain $\R_+ \times \R$ with eigenvalues:
\begin{equation}\label{eq:eigenvalues}
\lambda_1^\eps(p, u)=-\sqrt{-1/\hT'_\eps(p)}, \quad \lambda_2^\eps(p, u)=\sqrt{-1/\hT'_\eps(p)}.
\end{equation}
Moreover, the two associated characteristic fields are genuinely nonlinear on the domain $\R_+ \times \R$.
We denote $\mathcal{L}_i^\eps=(\hS_i^\eps, \hR_i^\eps)$ the Lax (shock and rarefaction) curves of the $i^{th}$ family associated with the eigenvalue $\lambda^i_\eps$. In the $(p, u)$ variables, the curves emanating from the left state $\hU^\ell = (p^\ell,u^\ell)$ are represented in Figure~\ref{fig:laxcurves} and read as follows:
\begin{align} \label{eq:laxcurves1}
\mathcal{S}_1^\eps(\sigma)(\hU^\ell): & \quad 
u=u^\ell-\sqrt{-(\hT_\eps(p)-\hT_\eps(p^\ell))(p-p^\ell)}, \quad \sigma = p - p^\ell > 0, \notag\\
\mathcal{R}_1^\eps(\sigma)(\hU^\ell):& \quad u=u^\ell-\int_{p^\ell}^p \sqrt{-\hT'_\eps(\xi)} \, d\xi, \quad \sigma < 0,
\end{align}
\begin{align} \label{eq:laxcurves2}
\mathcal{S}_2^\eps(\sigma)(\hU^\ell):& \quad u=u^\ell-\sqrt{-(\hT_\eps(p)-\hT_\eps(p^\ell))(p-p^\ell)}, \quad \sigma < 0, \notag\\
\mathcal{R}_2^\eps(\sigma)(\hU^\ell):& \quad u=u^\ell+\int_{p^\ell}^p \sqrt{-\hT'_\eps(\xi)} \, d\xi, \quad \sigma > 0. 
\end{align}
\end{proposition}

\begin{figure}
    \centering
    \includegraphics[scale = 0.45]{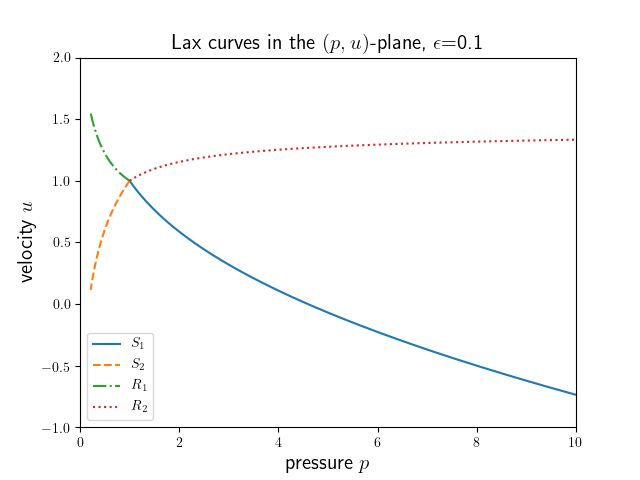}
    \includegraphics[scale = 0.45]{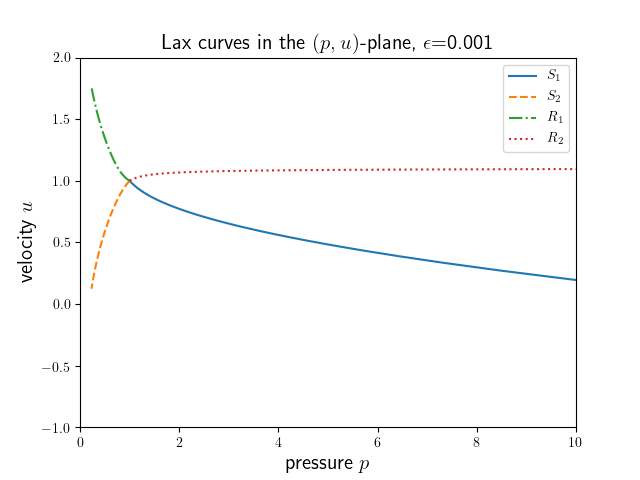}
    \caption{Lax curves emanating from the state $\hU^\ell = (1,1)$ are represented in the $(p,u)-$plane for $\eps=0.1$ (left) and $\eps = 0.001$ (right).
    We observe that the $\mathcal{R}_2$ rarefaction curve and the $\mathcal{S}_1$ shock curve tend to an horizontal line as $\eps \to 0$ which is characteristic of waves with infinite speed that connect states with same velocity but different pressures (Remark~\ref{rmk:inf-wave}).}
    \label{fig:laxcurves}
\end{figure}

\paragraph{Properties of the singular pressure.}
We observe that the behavior of the rarefaction and shock waves~\eqref{eq:laxcurves1}-\eqref{eq:laxcurves2} strongly depends on the size of $\hT_\eps'(p)$.
The next lemma provides bounds on $\hT_\eps' (p)$ according to the nature of the state.
\begin{lemma}\label{lem:pressure}
\begin{itemize}
   \item Let $\hU=(p,u)$ be a \textit{free state} \text{(F)} with $p \in [\pfm ,\pfp]$. Then:
   \begin{align}\label{eq:Tfree}
    C(\pfp) \leq  |\hT_\eps'(p)| \leq C(\pfm);
   \end{align}
    \item let $\hU=(p,u)$ be a \textit{congested state} \text{(C)} with  $p \in [\pcm ,\pcp]$. Then:
    \begin{align}\label{est:Tprime-cong}
    \eps^{\frac{1}{\gamma_c}}C(\pcp) \leq  |\hT_\eps'(p)| \leq \eps^{\frac{1}{\gamma_c}} C(\pcm),
   \end{align}
   where $C(\pfm), C(\pfp), C(\pcm), C(\pcp)$ are positive constants independent of $\eps$.
   \end{itemize}
\end{lemma}

\medskip

\begin{remark}\label{rmk:inf-wave}
Notice that the speed of waves connecting congested states tends to $\pm \infty$ as $\eps \to 0$ since $|\hT'_\eps(p)| \propto \eps^{\frac{1}{\gamma_c}}$ for $p \in [\pcm,\pcp]$.
This can be interpreted as a degeneracy of the $\mathcal{R}_2$ and $\mathcal{S}_1$ Lax curves into horizontal lines in the $(p,u)-$plane (see Figure~\ref{fig:laxcurves}).
At the level of the limit system~\eqref{eq:psystem-lim}, this leads to define waves with infinite speed. 
A wave of the first (resp. second) family with infinite speed is then located along a line that is parallel to the $x$-axis in the $x-t$ plane, and it is the limit of a sequence of waves of the first (resp. second) family for the $\eps$-soft congestion system~\eqref{eq:psystem-ep}. 
Such wave connects states with the same velocity but different values of the pressure.
The existence of these waves with infinite speed is the reason why the initial datum $(p_\text{in},u_\text{in})$ is in general not attained on $]-\infty,\overline{x}(0)]$, see Definition~\ref{entr-sol-def-2} and Remark~\ref{rmk:ID-hard}.\\
The so-called \textit{declustering waves} exhibited by Degond et al. in~\cite{degond2011} (Appendix A, Proposition 5) can be seen as a subclass of waves with infinite speed.
Let us point out however that such declustering waves associated with a contact discontinuity at the interface in~\cite{degond2011}, involving one state with free pressure (actually zero pressure since $\kappa=0$ in~\cite{degond2011}), cannot be observed in our context as we work in the case where $u(t,\bar x(t) - ) > u(t,\bar x(t) +)$, so that in our setting the interface is the limit of a 2-shock and not a contact discontinuity.
\end{remark}

\medskip
\bigskip
\paragraph{Front Tracking Algorithm.}
To establish the existence of weak entropy solutions to the $\eps$-soft-congested system \eqref{eq:psystem-ep}, we rely a front-tracking algorithm that we describe below.
\begin{definition}[$\rho$-Approximate front tracking solution, \cite{bressan2000}]
For $\rho > 0$ given, $\hU$ is a $\rho$-approximate front tracking solution to~\eqref{eq:psystem-ep} with initial datum
\begin{equation}\label{eq:init_data_ep}
    \hU(0, \cdot) = \hU^{\eps}_\text{in}(\cdot)
\end{equation}
if the following hold:
\begin{enumerate}[label=(\roman*)]
    \item $\hU$ is piecewise constant, with discontinuities occurring along finitely many straight lines in the $x-t$ plane. 
    Jumps can be either shocks (S) or rarefactions (R).
    \item Along each shock $x= x_\alpha(t)$, $\alpha \in S$, the values $\hU^\ell=(p_\ell, u_\ell) \doteq \hU(t,x_\alpha-)$ and $\hU^r=(p_r, u_r) \doteq \hU(t,x_\alpha+)$ are linked by the relation
    \[
    \hU^r = \mathcal{S}_{k_\alpha}(\sigma_\alpha)(\hU^\ell) \quad \text{for  } k_\alpha \in \{1,2\},
    \]
    Moreover the speed $\dot x_\alpha (t)$ of the shock wave satisfies
    \[
    |\dot{x}_\alpha - \lambda_{k_\alpha} (\hU^r, \hU^\ell)| \leq \rho,
    \]
    where $\lambda_{k} (\hU^r, \hU^\ell)$ are the eigenvalues of the averaged matrix $A(\hU^r,\hU^\ell)$:
    \[
    A(\hU^r,\hU^\ell) \doteq \int_0^1 DF(\theta \hU^r + (1-\theta)\hU^\ell) d\theta, 
    \qquad F(\hU) = \begin{pmatrix} - u\\ \hP_\eps(\tau) \end{pmatrix},
    \]
    where $D=(\de_\tau, \de_u)$.
    \item Along each rarefaction $x= x_\alpha(t)$, $\alpha \in R$, one has
    \[
    \hU^r = \mathcal{R}_{k_\alpha}(\sigma_\alpha)(\hU^\ell) \quad \text{for } k_\alpha \{1,2\}, 
    \]
    where $\sigma_\alpha = p_r -p_\ell \in ]0,\rho].$
    Moreover
    \[
    |\dot{x}_\alpha - \lambda_{k_\alpha} (\hU^r)| \leq \rho.
    \]
    \item At the initial time $\|\hU(0,\cdot) - \hU_\text{in}^\eps(\cdot)\|_{L^1(\R)} \leq\rho.$
\end{enumerate}
\end{definition}     
\paragraph{Definition of the algorithm~\cite{baiti2012}.}
At time $t=0$, we approximate the initial datum $\hU^0_\eps$ by a piecewise constant function with a finite number of jumps. 
At each discontinuity point, we solve the corresponding Riemann problem. Every generated rarefaction wave with strength $|\sigma|> \rho$ is partitioned in small (entropy violating) discontinuities with speed equal to the characteristic speed of the state at the right.\\
The fronts are prolonged until two of them interact. At that time, we solve (approximately) the emerging Riemann problem and so on. For interaction times $t >0$, we always partition an outgoing rarefaction wave, except when one of the incoming waves is a rarefaction of the same characteristic family. 
In that latter case, the outgoing rarefaction is substituted by a single jump of the same strength, with speed equal to the characteristic speed of the state at the right.

\subsection{Interaction estimates}\label{sec:interaction}
To study solutions constructed via the algorithm described above, it is necessary to provide estimates of the difference between the strengths of the corresponding incoming and outgoing fronts.
The next lemma states that, away from the interface, the classical interaction estimates hold independently of $\eps$.
\begin{lemma}[Interactions between two waves away from the interface]
Let a 2-wave of strength $\sigma'$ connecting free states (resp. congested states) interact with a 1-wave of strength $\sigma''$ connecting free states (resp. congested states). 
The outcome is a 1-wave of strength $\sigma_1^+$ connecting free states (resp. congested states) and a 2-wave of strength $\sigma_2^+$ connecting free states (resp. congested states).
Moreover, the following inequality holds
\begin{equation}
    |\sigma_1^+ - \sigma''| + |\sigma_2^+ - \sigma'| \leq C |\sigma' \sigma''|,
\end{equation}
for some $C > 0$ independent of $\eps$.
\end{lemma} 
Similar estimates can be obtained for two incoming waves of the same family, we refer to \cite{bressan2000} [Chapter 7, Lemma 7.2] for a complete statement and a detailed proof.\\
The next lemma is concerned with interactions of small waves with the interface. 
\begin{lemma}[Interactions with the interface]
\begin{itemize}
    \item {\textit Small wave coming the left congested domain}: Let a 2-wave of strength $\sigma'$ connecting congested states interact with the interface of strength $\bar \sigma''$. It results a 1-wave of strength $\sigma_1^+$ connecting congested states and a 2-shock interface of strength $\bar \sigma_2^+$. 
    Moreover, the following inequalities hold
    \begin{equation}
    |\sigma_1^+| + |\bar\sigma_2^+ - \bar\sigma''| \leq C |\sigma'|.
    \end{equation}
    \item {\textit Small wave coming the right free domain}: Let a 1-wave (or a 2-wave) of strength $\sigma''$ connecting free states interact with the interface of strength $\bar\sigma'$. It results a 1-wave of strength $\sigma_1^+$ connecting congested states and a 2-shock interface of strength $\bar\sigma_2^+$.
    Moreover, the following inequality holds
    \begin{equation}
    |\sigma_1^+| + |\bar\sigma_2^+ - \bar\sigma'| \leq C |\sigma''| .
    \end{equation}
\end{itemize} 
\end{lemma} 
\medskip
\begin{remark}
In the case of an interaction between the interface and a wave connecting free states, observe that no wave is reflected in the free domain. The outgoing small wave is indeed a wave connecting congested states. 
\end{remark}

\subsection{Glimm functional}\label{sec:Glimm}
In view of our hypotheses, the strength of the initial 2-shock interface is only sligthly perturbed by the interactions with the small waves coming from the left (i.e. from the congested domain) and the right (i.e. from the free domain).  
We define $\bar x(t) = \bar x^{\eps,\rho}(t)$ the position of the interface at time $t$ and $\bar \sigma(t) = \bar \sigma^{\eps,\rho}(t)$ its strength. 
Namely, $\bar \sigma(t)$ is the difference at time $t$ between the pressure on the right/free side of the interface and the pressure on the left/congested side of the interface.
We recall the notation
\begin{equation*}
I_{1,t} \doteq ]-\infty, \bar x(t)], \quad I_{2,t} \doteq [\bar x(t), +\infty[ \ .
\end{equation*}
\begin{definition}[Glimm functional]
For $t > 0$, we define the Glimm functional $\Upsilon(t)$ as follows
\begin{equation}
    \Upsilon(t) \doteq \bar \sigma(t) + \sum_{x_\alpha \in I_{2,t}} |\sigma_\alpha| + k_c \sum_{x_\alpha \in I_{1,t}} |\sigma_\alpha|
    + k_{i,f} Q_{i,f}(t) + k_{f,f} Q_{f,f}(t) + k_{c,c} Q_{c,c}(t),
\end{equation}
where the functionals $Q_{i,f}, Q_{f,f}, Q_{c,c}$ are quadradic in terms of the strengths of the waves:
\begin{align*}
    Q_{i,f}(t)  = |\bar \sigma(t)| \sum_{x_\alpha \in I_{2,t}} |\sigma_\alpha|, \quad 
    Q_{f,f}(t)  = \sum_{x_\alpha,x_\beta \in I_{2,t}} |\sigma_\alpha \sigma_\beta|, \quad 
    Q_{c,c}(t)  = \sum_{x_\alpha,x_\beta \in I_{1,t}} |\sigma_\alpha \sigma_\beta|,
\end{align*}
and $k_c ,k_{i,f}, k_{f,f},k_{c,c}$ are suitable positive constants (independent of $\eps$).
\end{definition}
\begin{proposition}[Decrease of the Glimm functional]\label{eq:glimm-decrease}
Under the hypotheses of Theorem~\ref{thm:main}, we can find $k_c, k_{f,f}, k_{c,c}$ and $k_{i,f}> k_c$ sufficiently large such that
\begin{equation}\label{eq:dec-Glimm}
    \Upsilon(t) \leq \Upsilon (0) \quad \forall t \geq 0.
\end{equation}
\end{proposition}
\begin{lemma}[Weighted total variation]\label{lem:WTV}
Let us assume that we have constructed the piecewise constant solution $(t,x) \mapsto \hU^{\eps,\rho}(t,x)=(p^{\eps,\rho}, u^{\eps,\rho})(t,x)$ up to time $t>0$, by means of the wave front tracking algorithm.
Let us define its weighted total variation:
\begin{equation}\label{eq:weightTV}
    WTV_\eps(t) \doteq \sum_{i=1}^2 TV(p^{\eps,\rho}(t, \cdot), I_{i,t}^\eps) 
    +  \frac{1}{\eps^{\frac{1}{2\gamma_c}}}\,\text{TV}\big(u^{\eps,\rho},\, I^\eps_{1,t}\big)
    + \text{TV}\big(u^{\eps,\rho},\, I^\eps_{2,t}  \big).
\end{equation}
Then, there exist two constant $0<C_1 < C_2$, independent of $\eps$ such that
\begin{equation}\label{eq:control-WTV}
C_1 WTV_\eps(t) \leq \Upsilon(t) \leq C_2 WTV_\eps(t).    
\end{equation}
\end{lemma}
\begin{proof}[Sketch of the proof]
The equivalence between $TV(p^{\eps,\rho}(t, \cdot), \R)$ and $\Upsilon(t)$ is direct since all the quadratic terms in $\Upsilon$ can be absorbed in the linear ones.
For the total variation of the velocity, we write the Rankine-Hugoniot condition satisfied by any discontinuity:
\[
u_r - u_\ell 
= \pm \sqrt{-(v_r - v_\ell)(p_r - p_\ell)} 
= \pm \sqrt{-(\hT_\eps(p_r) - \hT_\eps(p_\ell))(p_r - p_\ell)}
= \pm |p_r - p_\ell| \sqrt{-\hT_\eps'(\hat p)},
\]
for some $\hat p$ between $p_r$ and $p_\ell$.
and observe that
\begin{equation*}
    \hT_\eps'(\hat p) = 
    \begin{cases}
    O(1) & \quad \text{if} \quad p_r,p_\ell \in [\pfm,\pfp],\\
    O(\eps^{\frac{1}{2\gamma_c}}) & \quad \text{if} \quad p_r,p_\ell \in [\pcm,\pcp].
    \end{cases}
\end{equation*}
\end{proof}
\bigskip
\begin{corollary}\label{cor:TV_u_cong}
Combining \eqref{eq:control-WTV} with \eqref{eq:dec-Glimm} yields the following inequality:
\begin{align*}
    \frac{1}{\eps^{\frac{1}{2\gamma_c}}}\,\text{TV}\big(u^{\eps,\rho},\, I^\eps_{1,t}\big) \le C_1^{-1} \Upsilon (0) \quad \text{i.e.} \quad \text{TV}\big(u^{\eps,\rho},\, I^\eps_{1,t}\big) \le C \eps^{\frac{1}{2\gamma_c}}.
\end{align*}
\end{corollary}
The above inequality will imply that in the limit $\eps \to 0$ the velocity is constant in the congested domain.
\begin{remark}
We need to ensure that the number of interactions with two outgoing waves of the same family is finite. In every strip $[t^*- \delta t, t^*] \times [a,b]$, the interactions in which there is more than one outgoing wave of the same family can occur only in the case of two interacting shocks of the same family  (resp. strength $\sigma_\alpha$, $\sigma_\beta$), as we may have to split an outgoing rarefaction of size $\sigma^+$. One can show however that in this case, the difference of the values of the Glimm functional after and before interaction
$\Delta \Upsilon (t) < 0$ for any $\eps>0$. Since from \eqref{eq:dec-Glimm} $\Upsilon (t) \le \Upsilon (0)$, this can only happen a finite number of times.
\end{remark}

\section{Existence of a BV approximate solution at $\eps$ fixed and convergence as $\eps \to 0$}\label{sec:limit}

\subsection{Some properties of the wave front tracking approximation}\label{sec:propWFT}
We provide below Lipschitz estimates in time.
\begin{lemma}[Lipschitz continuity in time]\label{lem:Lip_t}
Let $(t,x) \mapsto \hU^{\eps,\rho}(t,x)$ be the wave front tracking approximation constructed in the previous section.
Let $0 < s< t$, then there exists $L > 0 $ independent of $\eps$, such that
\begin{align}
\|p^{\eps,\rho}(t,\cdot) - p^{\eps,\rho}(s,\cdot)\|_{L^1(I_{1,t})}
& \leq \eps^{-\frac{1}{2\gamma_c}} L (t-s), 
\quad \|p^{\eps,\rho}(t,\cdot) - p^{\eps,\rho}(s,\cdot)\|_{L^1(I_{2,t})} \leq L (t-s),\label{eq:Lip-t-p}\\
\|u^{\eps,\rho}(t,\cdot) - u^{\eps,\rho}(s,\cdot)\|_{L^1(\R)}
& \leq  L (t-s),  \label{eq:Lip-t-u} \\
\|\tau^{\eps,\rho}(t,\cdot) - \tau^{\eps,\rho}(s,\cdot)\|_{L^1(I_{1,t})}
& \leq \eps^{\frac{1}{2\gamma_c}} L (t-s),
\quad \|\tau^{\eps,\rho}(t,\cdot) - \tau^{\eps,\rho}(s,\cdot)\|_{L^1(I_{2,s})} \leq  L (t-s).\label{eq:Lip-t-v}
\end{align}
\end{lemma}
\medskip

\begin{proof}[Idea of the proof]
\begin{itemize}
    \item \textit{Pressure estimate.} To simplify, let us assume first that there is single front connecting free states with $x_\alpha(t') > \bar x(t')$ for all $t' \in [s,t]$ (no interaction with the interface $\bar x(t')$), we observe that
    \[
    \int |p(t,x) - p(s,x)| dx 
    = |x_\alpha(t) - x_\alpha(s)| |\sigma_\alpha|.
    \]
    In the general case, summing the contributions of all the fronts, we have
    \begin{align*}
    \int_{I_{2,t}} |p(t,x) - p(s,x)| dx 
    \leq \sum_{\alpha \in A_1} |x_\alpha(t) - x_\alpha(s)| |\sigma_\alpha|
    + \sum_{\alpha \in A_2} |x_\alpha(t'_\alpha) - x_\alpha(s)| |\sigma_\alpha|
    \end{align*}
    where $A_1$ denotes the set of waves connecting free states such that $x_\alpha(t') > \bar x(t')$ for all $t' \in [s,t]$, i.e. which do not interact with the interface; $A_2$ denotes the set of waves connecting free states that interact with the interface for some time $t_\alpha' \in [s,t]$, i.e. $x_\alpha(t'_\alpha) = \bar x(t'_\alpha)$ for some $t'_\alpha \in [s,t]$.
    We then get the second inequality announced in~\eqref{eq:Lip-t-p}:
    \begin{align*}
     \int_{I_{2,t}} |p(t,x) - p(s,x)| dx 
    \leq  \sum_{\alpha \in A_1} |\dot x_\alpha| |\sigma_\alpha| (t-s) 
    + \sum_{\alpha \in A_2}|\dot x_\alpha| |\sigma_\alpha| (t'_\alpha-s)
    \leq L (t-s),
    \end{align*}
    since all the speeds $\dot x_\alpha$ are bounded uniformly with respect to $\eps$.\\
    Let us now consider the congested domain $I_{1,t}$,we have similarly
    \begin{align*}
    \int_{I_{1,t}} |p(t,x) - p(s,x)| dx 
    & \leq \sum_{\alpha \in A_1} |x_\alpha(t) - x_\alpha(s)| |\sigma_\alpha|
    + \sum_{\alpha \in A_2} |x_\alpha(t'_\alpha) - x_\alpha(s)| |\sigma_\alpha|\\
    & \quad + \sum_{\alpha \in A_3} |x_\alpha(t) - x_\alpha(t'_\alpha)| |\sigma_\alpha|,
    \end{align*}
    where $A_1$ denotes the set of waves connecting congested states such that $x_\alpha(t') < \bar x(t')$ for all $t' \in [s,t]$, i.e. which do not interact with the interface in the time interval $[s,t]$; $A_2$ denotes the set of waves connecting congested states that interact with the interface at some time $t_\alpha' \in [s,t]$, i.e. $x_\alpha(t'_\alpha) = \bar x(t'_\alpha)$ for some $t'_\alpha \in [s,t]$; $A_3$ denotes the set of waves connecting congested states that have been created from an interaction with the interface at some time $t'_\alpha$.
    Since $|\dot x_\alpha| \propto \eps^{-\frac{1}{2\gamma_c}}$, we deduce~\eqref{eq:Lip-t-p}.
    \item \textit{Velocity estimate.} The same reasoning can be applied by replacing $\sigma_\alpha= \Delta p$ by $\Delta u$. We use Lemma~\ref{lem:WTV} to bound the velocity differences $\Delta u$ by $\sup|\sqrt{-\hT_\eps'(p)}| |\sigma| \leq C \eps^{\frac{1}{2\gamma_c}}$ in domain $I_{1,t}$. 
    This compensates the singularity of the speed $\dot{x}_\alpha$ in the congested domain and yields~\eqref{eq:Lip-t-u}.
    \item \textit{Specific volume estimate.} In the same manner, we derive~\eqref{eq:Lip-t-v} by writing $|\Delta \tau| \leq \sup |\hT_\eps'(p)| |\sigma|$ and estimating $|\hT_\eps'(p)|$.
\end{itemize}
\end{proof}

\bigskip
\begin{remark}
The singularity $\eps^{-\frac{1}{2\gamma_c}}$ in the estimate~\eqref{eq:Lip-t-p} involving the pressure is the main originality of the model and the signature of the infinite speed of propagation in the congested domain (a similar estimate is derived by Colombo et al. in~\cite{colombo2016ARMA} for the low Mach limit).
From the standpoint of the mathematical analysis, this singularity prevents us to get the strong convergence of the pressure as $\eps$ goes to 0 (see Section~\ref{sec:lim-ep} below). 
\end{remark}

\bigskip
Lipschitz estimates in space are collected in the following.
\begin{lemma}[Lipschitz continuity in space]\label{lem:Lip_x}
Under the assumptions of Theorem~\ref{thm:main}, let $t > 0$ and $\hU^{\eps,\rho}(t) \in \mathcal{D}^{\delta_2,\eps}_t$ be the wave front tracking approximate solution.
For any $0 < x_1 < x_2$, the following estimates hold.
\begin{itemize}
 \item On the free side of the interface:
    \begin{align}
\big\| p^{\eps,\rho}(\cdot,\bar x_\eps(\cdot) +  x_1) - p^{\eps,\rho}(\cdot,\bar x_\eps(\cdot) + x_2)\big\|_{L^1(0,t)} & \leq C  (x_2 - x_1);\\
\big\| u^{\eps,\rho}(\cdot,\bar x_\eps(\cdot) +  x_1) - u^{\eps,\rho}(\cdot,\bar x_\eps(\cdot) + x_2)\big\|_{L^1(0,t)} & \leq C (x_2 - x_1);
    \end{align}
    \item On the congested side of the interface:
\begin{align}
\big\| p^{\eps,\rho}(\cdot,\bar x(\cdot) -  x_1) - p^{\eps,\rho}(\cdot,\bar x(\cdot) - x_2)\big\|_{L^1(0,t)}
& \leq C (x_2 - x_1), \\
\big\| u^{\eps,\rho}(\cdot,\bar x(\cdot) -  x_1) - u^{\eps,\rho}(\cdot,\bar x(\cdot) - x_2)\big\|_{L^1(0,t)}
& \leq C \eps^{\frac{1}{2\gamma_c}}  (x_2 - x_1), \label{eq:lip_x_u_cong}
\end{align}
\end{itemize}
for some positive constant $C$ independent of $\eps$.
\end{lemma}

\medskip
\begin{proof}[Idea of the proof] 
Following a similar idea as before, we can write for a single front $(\sigma_\alpha, x_\alpha)$ propagating in the free domain (without any interaction with the interface or other waves)
\begin{align}\label{eq:Lip_x_step0}
\int_0^t \big| p^{\eps,\rho} (\tau, \bar x(\tau) + x_1) - p^{\eps,\rho} (\tau, \bar x(\tau) + x_2) \big| d\tau
\leq |\hat t_1 - \hat t_2 | \big|\sigma_\alpha\big|,
\end{align}
where $\hat t_i$, $i=1,2$, denotes the time at which the front arrives at the position $\bar x(t_i) + x_i$. 
Let us assume that the front was located at $\bar x(t_0) + x_0$ at time $t_0$ and define the function $x \mapsto \hat t(x)$ for $x> 0$ such that it holds
\[
\bar x(t_0) + x_0 + (\hat t(x) - t_0) \dot x_\alpha 
= \bar x (\hat t) + x.
\]
We observe that
\[
\hat t'(x) = \dfrac{1}{\dot x_\alpha - \dot{\overline{x}}(\hat t)}.
\]
As a consequence, coming back to~\eqref{eq:Lip_x_step0}, we get
\begin{align*}
 \int_0^t \big| p^{\eps,\rho} (\tau, \bar x(\tau) + x_1) - p^{\eps,\rho} (\tau, \bar x(\tau) + x_2) \big| d\tau
\leq \dfrac{|x_1 - x_2|}{\inf_{\tau\in [0,t]} \{|\dot x_\alpha - \dot {\bar{x}}| \} } \big|\sigma_\alpha\big|.
\end{align*}
One can then show that the difference between the speed of propagation of waves connecting free states (i.e. with pressure in the interval $[\pfm,\pfp]$) and the speed of interface is controlled from below uniformly with respect to $\eps$, which allows to control the above quantity by $C |x_1-x_2|$, with $C$ a constant independent of $\eps$. \\
More generally, one can prove that 
\begin{align*} \displaystyle
\int_0^t \big| p^{\eps,\rho} (\tau, \bar x(\tau) + x_1) - p^{\eps,\rho} (\tau, \bar x(\tau) + x_2) \big| d\tau
\leq \dfrac{|x_1 - x_2|}{\inf_{\tau \in [0,t]} \{|\dot x_\alpha - \dot {\bar{x}}| \} } \sup_{x \in [x_1,x_2]} TV_{[0,t]} \big( p^{\eps,\rho}(\cdot, \bar x(\cdot) + x) \big),
\end{align*}
and we use the decrease of the Glimm functional to demonstrate that 
\[
\sup_{x \in [x_1,x_2]} TV_{[0,t]} \big( p^{\eps,\rho}(\cdot, \bar x(\cdot) +x ) \big) \leq C,\]
with $C$ independent of $\eps$ when $x_{1,2} > 0$ (i.e. on the free side of the interface).\\
On the congested side of the interface, we combine the unbounded speed of propagation of the congested waves (proportional to $\eps^{-\frac{1}{2\gamma_c}}$) and the uniform control of $\sup_{x \in [x_1,x_2]} TV_{[0,t]} \big( u^{\eps,\rho}(\cdot, \bar x(\cdot) - x ) \big)$ to deduce~\eqref{eq:lip_x_u_cong}, while $\sup_{x \in [x_1,x_2]} TV_{[0,t]} \big( p^{\eps,\rho}(\cdot, \bar x(\cdot) - x ) \big)$ is shown to be controlled by $C \eps^{-\frac{1}{2\gamma_c}}$.
Additional details are provided in~\cite{ancona2022}.
\end{proof}
\subsection{Existence of a BV solution $\hU^\eps=(p^\eps,u^\eps)$}\label{sec:lim-rho}
Let now explain how to pass to the limit $\rho\to 0$ and prove the first result announced in Theorem~\ref{thm:main}.

\medskip

Let $\rho = \rho_n$ decreasing to $0$ as $n \to +\infty$. 
For each $n \geq 1$ we have constructed a $\rho_n$-approximate solution,  $\hU^{\eps,\rho_n}$, of the Cauchy problem~\eqref{eq:psystem-ep}-\eqref{eq:init_data_ep}.
From the estimates of Section~\ref{sec:Glimm}, we ensure that $(\hU^{\eps,\rho_n})_n$ has a bounded total variation in space (for fixed $\eps$). 
On the other hand, the maps $t \mapsto \hU^{\eps,\rho_n}(t, \cdot)$ have been shown to be Lipschitz continuous with values in ${\bf L^1}(\R;\R^2)$ in Lemma~\ref{lem:Lip_t}, with a Lipschitz constant uniform with respect to $n$ (but dependent on $\eps> 0$).
We can then extract a subsequence (see Theorem 2.4 from~\cite{bressan2000}), still denoted $(\hU^{\rho_n,\eps})_n$, which converges in ${\bf L^1_{loc}}$ to some limit $\hU^{\eps}$ belonging to $\mathcal{D}^{\delta_2,\eps}_t$ for a.e. $t \geq 0$ and for some $\delta_2 \leq \delta_0$.\\
Initially, by Lemma~\ref{lem:Lip_t}, $\|\hU^{\eps,\rho_n}(0,\cdot) - \hU^\eps_{in}(\cdot)\|_{L^1} \to 0$ as $n \to +\infty$.
One can next follow the lines of~\cite{bressan2000} [Chapter 7.4] to verify that $\hU^\eps$ is a weak entropy solution of~\eqref{eq:psystem-ep} in the sense of Definition~\ref{entr-sol-def-1}.\\
Let us recall that we have defined in Section~\ref{sec:Glimm} the curve $t \mapsto \bar x^{\eps,\rho_n}(t)$ representing the position of the interface between the congested domain $I^{\eps,\rho_n}_{1,t}$ and the free domain $I^{\eps,\rho_n}_{2,t}$.
We have
\[
\dot{\bar x}^{\eps,\rho_n}(t)
= - \dfrac{u^{\eps,\rho_n}(t, \bar x^{\eps,\rho_n}(t) +)  -  u^{\eps,\rho_n}(t, \bar x^{\eps,\rho_n}(t) -)}
{ \hT_\eps(p^{\eps,\rho_n}(t, \bar x^{\eps,\rho_n}(t) +))  - \hT_\eps( p^{\eps,\rho_n}(t, \bar x^{\eps,\rho_n}(t) -))}.
\]
In view of the control of $\hU^\eps(t,\cdot)$ in $\mathcal{D}^{\delta_2,\eps}_t$, it is easy to observe that for $\eps$ small enough
\begin{equation}
\bar\lambda_2 - \delta_0 < \dot{\bar x}^{\rho_n,\eps}(t) < \bar\lambda_2 + \delta_0,
\end{equation}
$\bar \lambda_2$ being the speed of the reference $2$-shock interface defined in~\eqref{RH-lim} (which is also independent of $n$ and $\eps$).
This control ensures that there exists $\bar x^\eps \in W^{1,\infty}(\R_+)$ such that
\begin{align}
    \bar x^{\eps,\rho_n} &\to \bar x^{\eps} \quad \text{uniformly on any}~ [0,T], ~ T > 0,\\
    \dot{\bar x}^{\eps,\rho_n}(t) &\to \dot{\bar x}^{\eps}(t) \quad \text{a.e.}~ t \in [0,T].
\end{align}

\subsection{Convergence $\eps \to 0$}\label{sec:lim-ep}

Let us now achieve the proof of Theorem~\ref{thm:main} by passing to the limit with respect to the parameter $\eps$.
We split the proof in several steps.
\begin{itemize}
    \item First, let us observe that, with the same arguments as in Section~\ref{sec:lim-rho}, we ensure the existence of $\bar x \in W^{1,\infty}(\R_+)$ such that
    \begin{align}
    \bar x^{\eps} &\to \bar x \quad \text{uniformly on any}~ [0,T], ~ T > 0,\\
    \dot{\bar x}^{\eps}(t) &\to \dot{\bar x}(t) \quad \text{a.e.}~ t \in [0,T].
    \end{align}

    \item Using estimates~\eqref{eq:Lip-t-u} on the one hand and~\eqref{eq:dec-Glimm}-\eqref{eq:control-WTV} on the other hand, we can again extract a subsequence $(u^\eps)_\eps$ converging in ${\bf L^1_{loc}}$ towards some limit $u^*$.
    In the same manner, we ensure that $(\tau^\eps)_\eps$ converges (up to a subsequence) towards some limit $\tau^*$ in ${\bf L^1_{loc}}$.
    Regarding the pressure, we only infer that $(p^\eps)_\eps$ converges weakly-* in ${\bf L^\infty}(\R_+ \times \R)$.
    We also ensure that the limit $(p^*(t,\cdot),u^*(t,\cdot))$ belongs to $\mathcal{D}^{\delta_2}_t$ for any $t$. \\
    Passing to the limit in the sense of distributions in~\eqref{eq:psystem-ep}, we check that the following equations
    \begin{equation}\label{eq:limit-0}
    \begin{cases}
    \partial_t \tau^* -\partial_x u^* = 0, \\
    \partial_t u^* + \partial_x p^* = 0,
    \end{cases}
    \end{equation}
    hold in the sense of distributions.
    \item Let us now define the shifted variable:
    \begin{equation}
        \tilde\hU^\eps(t,x) \doteq \hU^\eps(t, x- \bar x(t) - \bar x^\eps(t)), \quad \forall \ t>0,x\in \R,
    \end{equation}
    so that $\tilde\hU^\eps(t,\bar x(t)) = \hU^\eps(t, \bar x^\eps(t))$.
    In view of the previous arguments, we ensure that
    \begin{align*}
        \tilde{u}^\eps (t,\cdot) \to u^*(t,\cdot) \quad \text{strongly in}~{\bf L^1}(I_{1,t})~\text{ for a.a. } t \geq 0,
    \end{align*}
    and from Corollary~\ref{cor:TV_u_cong}:
    \begin{align*}
        TV (u^*(t),I_{1,t}) 
        & \leq \liminf_{\eps \to 0} \ TV(\tilde{u}^\eps(t), I_{1,t}) \\
        & \leq \liminf_{\eps \to 0} \ TV(u^\eps(t), I_{1,t}^\eps) = 0.
    \end{align*}
    Consequently, $u^*(t, \cdot)$ is constant on the congested domain $I_{1,t}$, for a.a. time $t \geq 0$:
    \begin{equation}
    u^*(t,x) = u_c(t) \quad \forall \ x \in I_{1,t}. 
    \end{equation}
    Similarly, we show that
    \begin{equation}
     \tilde{\tau}^\eps (t,\cdot) \to \tau^*(t,\cdot) \quad \text{strongly in}~L^1(I_{1,t})~\text{ for a.a. } t \geq 0.
    \end{equation}
   and $ \tau^*(t,x) \equiv \text{const} = 1$ for all $x \in I_{1,t}$.
    \item We now use the Lipschitz continuity estimates in space derived in Lemmas~\ref{lem:Lip_x} to infer the strong convergence of the traces of the different variables from both sides of the interface.
    On the free side of the interface, we have
    \begin{align*}
        \tilde{p}(\cdot, \bar x(\cdot) + ) \to p^*(\cdot, \bar x(\cdot) + ) \quad & \text{in} \quad {\bf L^1_{loc}}(\R^+), \\
        \tilde{u}(\cdot, \bar x(\cdot) + ) \to u^*(\cdot, \bar x(\cdot) + ) \quad & \text{in} \quad {\bf L^1_{loc}}(\R^+), \\
        \tilde{\tau}(\cdot, \bar x(\cdot) + ) \to \tau^*(\cdot, \bar x(\cdot) + ) \quad & \text{in} \quad {\bf L^1_{loc}}(\R^+),
    \end{align*}
    while on the congested side of the interface:
    \begin{align*}
        \tilde{u}(\cdot, \bar x(\cdot) - ) \to u^*(\cdot, \bar x(\cdot) - ) = u_c(\cdot) \quad & \text{in} \quad {\bf L^1_{loc}}(\R^+), \\
        \tilde{\tau}(\cdot, \bar x(\cdot) - ) \to \tau^*(\cdot, \bar x(\cdot) - ) = 1 \quad & \text{in} \quad {\bf L^1_{loc}}(\R^+),
    \end{align*}

    \item   For a.a $t$, the value $u_c(t)$ is calculated by passing to the limit in the Rankine-Hugoniot relation
    \begin{equation*}
        \tilde u^\eps(t, \bar x(t)-) = \tilde u^\eps(t, \bar x(t) +) + \dot{\overline x}^\eps(t) \big(\tilde \tau^\eps(t,\bar x(t)+) -1\big),
    \end{equation*}
    and using the convergence of the traces from both sides of the interfaces:
    \begin{equation}\label{eq:uc_0}
        u_c(t) = u^*(t, \bar x(t) +) + \dot{\overline x}(t) \big(\tau^*(t,\bar x(t)+) -1\big).
    \end{equation}
    \item Let us now discuss the pressure.
    First, we observe from the estimates of the previous sections that $\tilde p^\eps$ satisfy uniform estimates in the free domain $I_{2,t}$, so that
    \[
    \tilde p^\eps \to p^* \quad \text{strongly in}~ \mathbf{L^1_{loc}}\Big(\underset{t\geq 0}{\bigcup} \{t\} \times I_{2,t}\Big).
    \]
    This strong convergence in the free domain, combined with the fact that $\tau \equiv 1$ in the congested domain $I_{1,t}$ allows us to identify the limit specific volume: $\tau^* = \hT(p^*)$ with $\hT$ defined by~\eqref{eq:T-lim}, i.e. the following equations hold in the sense of distributions
    \begin{equation}\label{eq:limit-1}
    \begin{cases}
    \partial_t \hT(p^*) -\partial_x u^* = 0, \\
    \partial_t u^* + \partial_x p^* = 0.
    \end{cases}
    \end{equation}
    \item We can determine the pressure $p^*(t,\cdot)$ in the congested domain $I_{1,t}$. 
    Looking at the limit momentum equation and using the fact that $u^*(t,\cdot)$ is constant on $I_{1,t}$, we deduce that $p^*(t,\cdot)$ is affine.
    On the other hand, since $\hU^*(t,\cdot) \in \mathcal{D}^{\delta_2}_t$, we have $p^*(t,x) \in ]p_{0,1}-\delta_2,p_{0,1}+\delta_2[$ for all $x \in I_{1,t}$.
    Hence $p^*(t,\cdot)$ is constant in $I_{1,t}$: $p^*(t,x) = p_c(t)$ for all $x \in I_{1,t}$.
    The value $p_c(t)$ is finally calculated from the Rankine-Hugoniot condition,  since we know that $(p^*, u^*)$ is a distributional solution of~\eqref{eq:limit-1}.
    Hence,
    \begin{equation}\label{eq:pc_0}
        p^c(t) = p^*(t,\bar x(t)+) + \dot{\bar x}(t) \big( u_c(t) - u^*(t, \bar x(t)+) \big).
    \end{equation}
    \item Since $p^*(t,\cdot)$ is constant in the congested domain, the limit momentum equation tells us that $t \mapsto u_c(t)$ is constant, i.e.:
    \[
    u^c(t) = u_c(0) \quad \forall \ t > 0.
    \]
    As a consequence, we get from \eqref{eq:uc_0}-\eqref{eq:pc_0}:
    \begin{align*}
     p^c(t) & = p^*(t,\bar x(t)+) + \dot{\bar x}(t) \big( u^c(0) - u^*(t, \bar x(t)+) \big), \\
    \dot{\bar x}(t) & = \dfrac{u^c(0) - u^*(t, \bar x(t)+)}{\tau^*(t,\bar x(t) +) - 1}. 
    \end{align*}
    In particular, one can check that $p^c(0)= \text{In}(p_\text{in})$ according with~\eqref{eq:in-function} introduced in Definition~\ref{entr-sol-def-2}.
\end{itemize}
This achieves the proof of Theorem~\ref{thm:main}.

\section{Extension to more general reference solutions}\label{sec:concl}
We presented the general strategy of construction of weak-entropy BV solutions in the case of an initial datum $\hU_\text{in} = (p_\text{in}, u_\text{in}) \in \mathcal{D}_0^{\delta, \eps}$ as defined in \eqref{eq:set-solns-eps}. As pointed out before, the initial configuration is a small BV perturbation of a reference solution that is given by a 2-shock congested/free interface, see Figure \ref{fig:sol-ref}.
However, the general strategy applies to more general configurations, with more than one single free/congested interface.
\subsection*{The case of two non-interacting discontinuity interfaces}
The reference solution for the limit (hard-congested) system \eqref{eq:psystem-lim} is piece-wise constant with two jump discontinuities: the external states are \textit{free} and the middle one is \textit{congested}. Then, given  $\overline x_{1,0} < \overline  x_{2,0}$, we let \begin{equation}
    \label{eq:p-ref-1}
\begin{aligned}
p^\text{ref}(x) & \doteq  p_{1,0}\mathbf{1}_{]-\infty,\, \overline x_{1,0}[}(x) + p_{2,0}\mathbf{1}_{]\overline x_{1,0},\, \overline x_{2,0}[}(x) + p_{3,0}\mathbf{1}_{]\overline  x_{2,0},\,+\infty[}(x), \\
u^\text{ref}(x) & \doteq u_{1,0}\mathbf{1}_{]-\infty,\, \overline x_{1,0}[}(x) + u_{2,0}\mathbf{1}_{]\overline x_{1,0},\, \overline x_{2,0}[}(x) + u_{3,0}\mathbf{1}_{]\overline x_{2,0},\, +\infty[}(x),
\end{aligned}
\qquad x\in\R,
\end{equation}
where
\begin{equation}
\label{eq:pu-lim-ref}
 p_{1,0}, \, p_{3,0} <\kappa < p_{2,0},
\qquad \qquad
u_{1,0} > u_{2,0} >  u_{3,0},
\end{equation}
with
\begin{equation}\label{eq:RH-lim-ref}
p_{2,0} - p_{1,0} = 
\dfrac{(u_{2,0} - u_{1,0})^2}{\hT^{i}(p_{1,0})-1}, \qquad   
p_{3,0} - p_{2,0} = -\dfrac{(u_{3,0} - u_{2,0})^2}{\hT^i(p_{3,0})-1}.
\end{equation}
\begin{figure}
	\centering \includegraphics[scale=0.65]{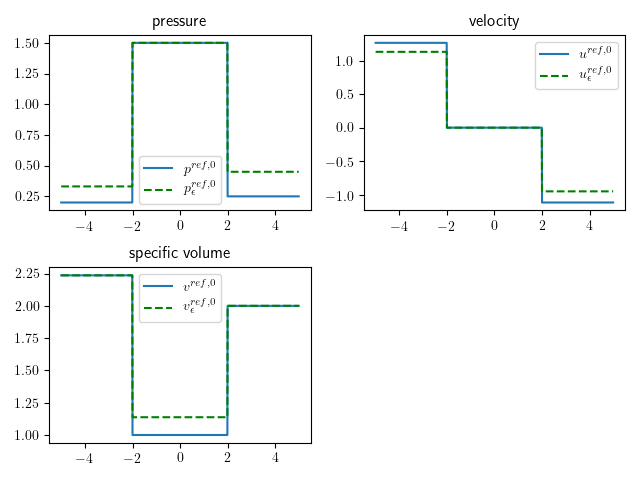}
 	\caption{Reference solutions with two non-interacting interfaces, $\eps = 0.2$, $\kappa =1$.}
 	\label{fig:ref_sol}
 \end{figure}
\noindent This way the left interface is a 1-shock wave and the right one is a 2-shock wave. Similarly to \eqref{eq:RH-eps-ref}, the initial data of the approximate model is a suitable approximation $(p^{\text{ref}, \eps}, u^{\text{ref}, \eps})$ of $(p^\text{ref}, u^\text{ref})$: we can choose in particular (see Figure \ref{fig:ref_sol})
\begin{align*}
    p_{1,0}^\eps = p_{1,0},\quad\
    p_{3,0}^\eps = p_{3,0}, \qquad u_{1,0}^\eps=u_{1,0} > u_{2,0} > u_{3,0}^\eps=u_{3,0},\,
\end{align*}
where $(p_{2,0}^\eps, u_{2,0}^\eps)$ is obtained by requiring that it is both: the right state of a 1-shock with left state given by $(p_{1,0}^\eps, u_{1,0}^\eps)$; the left state of a 2-shock with right state $(p_{3,0}^\eps, u_{3,0}^\eps)$.
\paragraph{Result}[see \cite{ancona2022}]
For any initial datum $\hU_\text{in}^\eps =(p_\text{in}^\eps, u_\text{in}^\eps)$ that is a \textit{small} BV perturbation of the above $(p^{\text{ref},\eps}, u^{\text{ref},\eps})$, there exists an entropy weak solution $\hU^\eps=(p^\eps, u^\eps)$ to the (approximate) soft-congested system \eqref{eq:psystem-ep} (with a suitably adapted singular pressure $\hP_\eps (\tau))$ and two (Lipschitz) interfaces ${\overline x}_i^\eps(t), \, i \in \{1,2\}$. Moreover, for any initial datum $\hU^\text{in}$ such that $ \hU_\text{in}^\eps \to \hU^\text{in}$ in ${{\bf L^1_{loc}}}$, there exists a global-in-time entropy weak solution $\hU^*=(p^*, u^*)$ to the (limit) hard-congested system \eqref{eq:psystem-lim} with (limiting) interfaces ${\overline x}_i(t), \, i \in \{1,2\}$.
In particular, the solution $(p^*, u^*)$ inside the congested domain $\big\{(t,x)\,:\ \overline x_1(t)<x<\overline x_2(t),\ t\geq 0\big\}$ has different properties with respect to the case of one single interface:
 \begin{equation}
      \label{eq:p-u-cong-value}
      \begin{aligned}
       p^*(t,x)&=
       \left(\frac{\overline{x}_2(t)-x}{\overline{x}_2(t)-\overline{x}_1(t)}\right) p^{c}_1(t)
+ \left(\frac{x-\overline{x}_1(t)}{\overline{x}_2(t)-\overline{x}_1(t)}\right)p^{c}_2(t)\,,
       \\
       \noalign{\medskip}
       u^*(t,x)&=u^{c}(t)= u_\text{in}^c-
       \int_0^t
       \frac{\,p^{c}_2(s)-p^{c}_1(s)\,}
      {\overline{x}_2(s)-\overline{x}_1(s)}~ds\,,
      \end{aligned}
      \end{equation}
      where $p^{c}_i(t), \, i \in \{1,2\}$ are given by the Rankine-Hugoniot conditions as in~\eqref{eq:RH-lim-ref}:
      \begin{align*}
         p^{c}_1(t)&=p^*(t, \overline{x}_1(t)-)+
        \frac{(u^{c}(t)-u^*(t, \overline{x}_1(t)-))^2}{\hT^{i}(p^*(t, \overline{x}_1(t)-))-1}; \\
        \noalign{\smallskip}p^{c}_2(t)&=p^*(t, \overline{x}_2(t)+)+\frac{(u^*(t, \overline{x}_2(t)+)-u^{c}(t))^2}{\hT^{i}(p^*(t, \overline{x}_2(t)+))-1}.
      \end{align*}
      %
      %

\noindent A striking point of the above result is that the construction of the entropy solution $\hU^*$ to the limit system \eqref{eq:psystem-lim}, as a suitable limit of solutions $\hU^\eps$ of the approximate one \eqref{eq:psystem-ep}, relies on a linearization 
of the singular pressure $\hP_\eps(\tau)$, which is crucial to obtain the desired uniform bounds on the the Glimm functional.
\subsection*{The case of two interacting discontinuity interfaces}
This is the case where the reference solution for the limit system \eqref{eq:psystem-lim} has again two jumps but the external states are congested and the middle one is free. 
More precisely, the left (congested) extreme and the middle (free) state are connected by a 2-shock interface, while the middle (free) and the right (congested) extreme are connected by a 1-shock interface. 
Being the 2-shock interface initially located at the left of the 1-shock interfaces, such interfaces will interact at some time $t^*>0$. 
The interaction between two interfaces is a delicate issue since it involves interactions of \textit{congested states with different velocities}. 
Once the interfaces connecting such external congested states with different velocities 
interact, they generate two shocks with an intermediate state having \textit{unbounded pressure},
that blows up in the limiting system with a precise rate given by the singular pressure~$\hP_\eps$. 
This interesting phenomenon is the reason why in this case we can provide an entropy weak solution $\hU^*(t)$ to \eqref{eq:psystem-lim} that is defined only on a \textit{finite} interval of time.
\subsection*{The general case}
Consider now the case of a reference solution with three interfaces given by two interacting discontinuity interfaces (say a 2-shock interface $\overline x_1$ on the left of a 1-shock interface $\overline x_2$ as in the previous case) followed by a 2-shock interface $\overline x_3$. 
After the interaction between the shocks located at $\overline x_1$, $\overline x_2$ occurs, two new discontinuities emerge, say a 1-shock $\overline x_2'$ on the left of a 2-shock $\overline x_1'$ that travels with a faster speed than the one of the 2-discontinuity $\overline x_3$. Therefore the two discontinuities $\overline x_2'$, $\overline x_3$ will eventually interact generating a 1-rarefaction located between two lines $\overline x_4, \overline x_5$, and a 2-discontinuity $\overline x_{1,3}$.
Notice that the 1-rarefaction generated by such an interaction connects two congested states: a left state with unbounded pressure (along $\overline x_4$) with a right state with bounded pressure (along $\overline x_5$). 
On the other hand, all waves connecting congested states in the limit system travel with infinite speed. Hence, the front lines $\overline x_1'$, $\overline x_2'$, $\overline x_4$, $\overline x_5$ will all converge to parallel lines to the x-axis in the solution of the limit system and the congested region of unbounded pressure between $\overline x_1'$, $\overline x_2'$, $\overline x_4$, as well as the rarefaction fan between $\overline x_4$, $\overline x_5$, coalesces in a line of measure zero in $x$-$t$ plane. As a result, in this case we can indeed obtain a solution for the limit system which is globally defined in time.
We give a representation of the wave interactions in Figure~\ref{fig:3int}.
\begin{figure}
    \centering
    \includegraphics{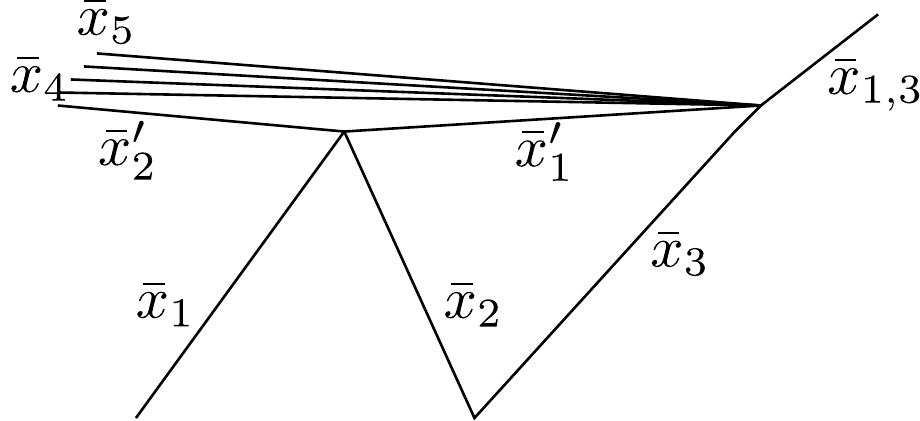}
    \caption{Sketchy representation of the case of three interfaces in the $(x-t)$-plane for small $\eps$.}
    \label{fig:3int}
\end{figure}

\medskip
Our conjecture is that, for a general reference solution $\hU^\text{ref}=(p^\text{ref}, u^\text{ref})$ with a finite number of discontinuity free/congested (resp. congested/free) interfaces, there are two possibilities:
\begin{enumerate}
    \item if at least one of the extreme states is {free}, then there exists a global-in-time BV entropy weak solution;
    \item if the extreme values are both congested, then there exists a BV entropy weak solution defined on a \textit{finite} time interval.
\end{enumerate}

\bibliography{biblio}
\end{document}